\documentclass[11pt,reqno]{amsart}

\usepackage[compact]{titlesec}
\titleformat{\section}[hang]{\normalfont\bfseries\filcenter}{\thesection.}{0.5em}{}
\titlespacing{\section}{0pt}{18pt}{5pt} 

\makeatletter
\def\thm@space@setup{%
  \thm@preskip=8pt  
  \thm@postskip=8pt 
}
\makeatother

\usepackage{amsmath}
\usepackage{amsthm}
\usepackage[all]{xy}
\usepackage{color}
\usepackage{amssymb}
\usepackage{graphicx}
\usepackage{epsfig}
\theoremstyle{plain}
\newtheorem{theorem}{Theorem}[section]
\newtheorem*{theorem*}{Theorem}

\newtheorem{lemma}[theorem]{Lemma}
\newtheorem{corollary}[theorem]{Corollary}

\newtheorem{conjecture}[theorem]{Conjecture}

\newtheorem{Important Facts}[theorem]{Important Facts}

\theoremstyle{definition}
\newtheorem{definition}[theorem]{Definition}

\newtheorem{remark}[theorem]{Remark}

\newtheorem*{acknowledgements}{Acknowledgements}

\def\int{\text{interior}}

\def\reals {\hbox {\rm {R \kern -2.8ex I}\kern 1.15ex}}
\def\integers {\hbox {\rm { Z \kern -2.8ex Z}\kern 1.15ex}}
\def\naturals {\hbox {\rm {N \kern -2.8ex I}\kern 1.20ex}}
\def\rationals {\hbox {\rm { Q \kern -2.2ex l}\kern 1.15ex}}
\def\hyp {\hbox {\rm {H \kern -2.7ex I}\kern 1.25ex}}

\newcommand{\NN}{\mathcal{N}}

\newcommand{\ssm}{\smallsetminus}
\newcommand{\bea}{\begin{eqnarray*}}
\newcommand{\eea}{\end{eqnarray*}}

\catcode `\@=11
\long\def\@savemarbox#1#2{\global\setbox#1\vtop{\hsize\marginparwidth 
  \@parboxrestore\tiny\raggedright #2}}
\catcode`\@=12

\def\strutdepth{\dp\strutbox}
\def \ss{\strut\vadjust{\kern-\strutdepth \sss}}
\def \sss{\vtop to \strutdepth{
\baselineskip\strutdepth\vss\llap{$\diamondsuit\;\;$}\null}}

\def\strutdepth{\dp\strutbox}
\def \sst{\strut\vadjust{\kern-\strutdepth \ssss}}
\def \ssss{\vtop to \strutdepth{
\baselineskip\strutdepth\vss\llap{$\spadesuit\;\;$}\null}}

\def\strutdepth{\dp\strutbox}
\def \ssh{\strut\vadjust{\kern-\strutdepth \sssh}}
\def \sssh{\vtop to \strutdepth{
\baselineskip\strutdepth\vss\llap{$\heartsuit\;\;$}\null}}

\begin{document}

\title[] {Bridge distance and plat projections}

\author{Jesse Johnson}

\author{Yoav Moriah}


\date{\today}


\address{Department of Mathematics \\
Oklahoma State University \\
Stillwater, OK 74078}
\email{jjohnson@math.okstate.edu}

\address{Department of Mathematics \\
Technion \\
Haifa, 32000 Israel}
\email{ymoriah@tx.technion.ac.il}

\begin{abstract} 
We calculate the bridge distance for  $m$-bridge knots/links in the $3$-sphere with  sufficiently complicated $2m$-plat projections. In particular we show that if the underlying braid of the plat has $n - 1$ rows of twists and all its exponents have absolute value greater than or equal to three then the distance of the bridge sphere 
is exactly $\lceil n/(2(m - 2)) \rceil$, where $\lceil x \rceil$ is the smallest integer greater than or equal to $x$.  As a corollary, we conclude that  if such a diagram has more than $4m(m-2)$ rows then the bridge sphere  defining the plat projection is the unique minimal bridge sphere for the knot.

\end{abstract}

\maketitle

\section{Introduction}

Let $K \subset S^3$ be an $m$-bridge link (possibly with one component) in an $m$-bridge position with respect to a bridge sphere $\Sigma$. Then $K \subset S^3$ has a plat projection as indicated in Figure \ref{plat} below.  Each box is marked by  $a_{i,j} \in \mathbb{Z}$  denoting  $a_{i,j}$ half twists,  where $1 \leq i \leq n - 1$ and $1 \leq j \leq m$ when $i$ is even, and  $1 \leq j \leq m - 1$ when $i$ is odd. The coefficients $a_{i,j}$ will be defined more precisely in Section \ref{plats}. We refer to $n$ as the {\it height} of the plat and to $m$ as the {\it width} of the plat.

\begin{definition}
\label{tightly twisted} 
A $2m$-plat will be called {\it highly twisted} if $|\, a_{i,j} |\, \geq 3$ for all $i, j$. Similarly, a link with a highly twisted plat projection will be called a {\it highly twisted link}.
\end{definition}

Associated with every bridge sphere $\Sigma$ for a link (or knot)  $K \subset S^3$ is the {\it bridge sphere distance} denoted by $d_{\Sigma}(K)$, as defined in~\cite{BS-bridges} and in Section~\ref{distance} below. Let $\lceil x \rceil$ be the ceiling function, which is equal to the smallest integer greater than or equal to $x$. We prove the following:

\begin{theorem}
\label{main theorem}
If $K \subset S^3$ is an $m$-bridge link with a highly twisted $n$-row, $2m$-plat projection for $m \geq 3$ then the distance $d(\Sigma)$ of the induced bridge surface $\Sigma$ is exactly $d(\Sigma) =  \lceil n /(2(m - 2))  \rceil$.
\end{theorem}

\begin{figure}[ht]
\includegraphics[width=3.5in]{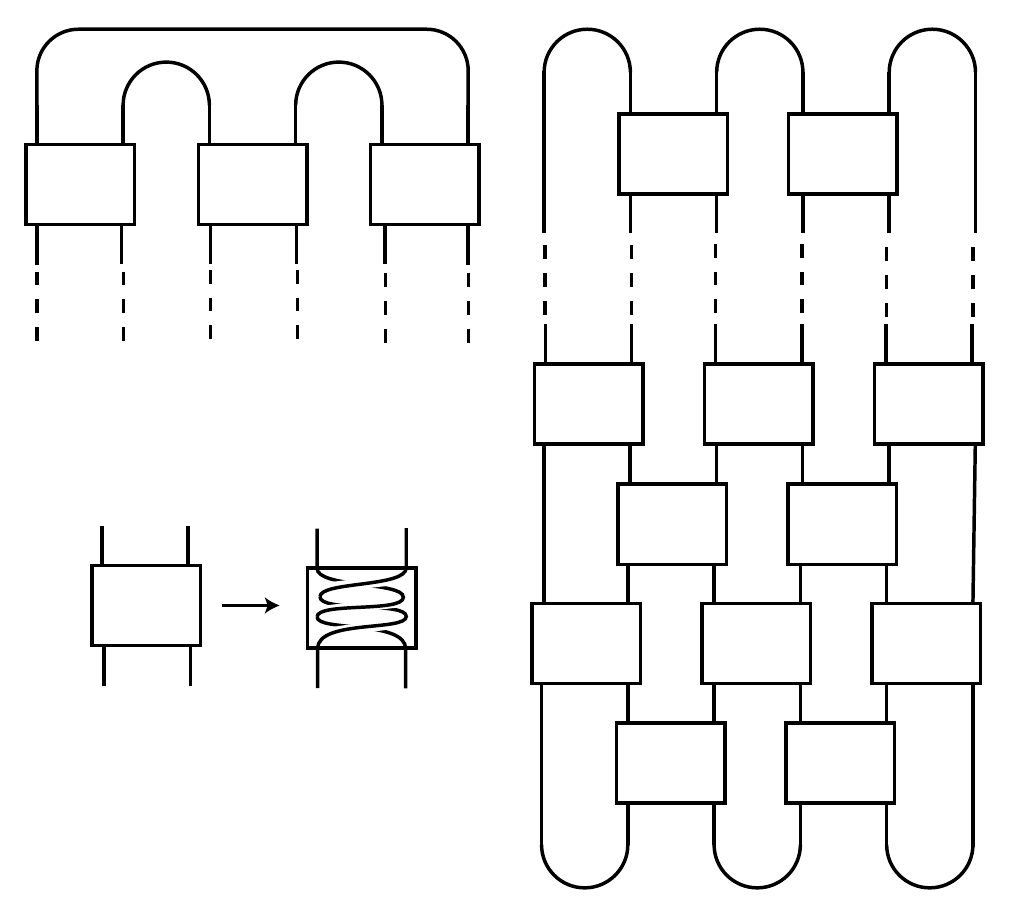}
\put(-220,72){3}
\put(-93,33){$a_{1,1}$}
\put(-50,33){$a_{1,2}$}
\put(-115,64){$a_{2,1}$}
\put(-71,64){$a_{2,2}$}
\put(-28,64){$a_{2,3}$}
\put(-93,94){$a_{3,1}$}
\put(-50,94){$a_{3,2}$}
\put(-115,124){$a_{4,1}$}
\put(-71,124){$a_{4,2}$}
\put(-28,124){$_{4,3}$}
\put(-96,186){$_{n-1,1}$}
\put(-53,186){$_{n-1,2}$}
\put(-245,180){$_{n-1,1}$}
\put(-201,180){$_{n-1,2}$}
\put(-158,180){$_{n-1,3}$}
\caption{A plat projection of a $3$-bridge knot}
\label{plat}
\end{figure}
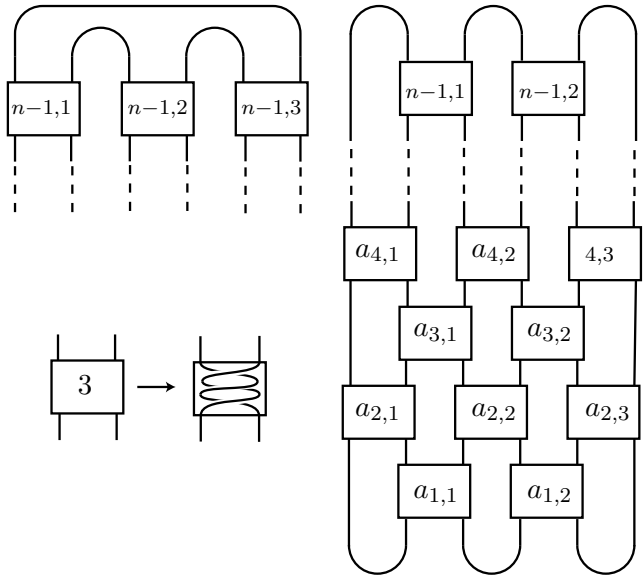

The bridge sphere distance of a knot in $S^3$ is a measure of the ``complexity''  of the gluing map between the boundary spheres of the rational tangles above and below the bridge sphere. One would expect intuitively that a braid with a more ``complicated''  diagram would induce a more complicated  gluing map. However, there were no previously known ways to determine distance from a knot projection. Results on distances of  Heegaard splittings in the case of closed surfaces or bridge surfaces for knot spaces typically split into two kinds. One  presents either lower or upper bounds, for example,  Evans~\cite{Evans}, Blair-Tomova-Yoshizawa~\cite{BTY}, Tao Li~\cite{Li-distance},  Ichiahra-Saito~\cite{Ich-Sa-bridges} and Lustig-Moriah~\cite{LuMo}. The other kind presents for given integers $n$ a manifold/knot space with distance 
$n$, for example, Ido-Jang-Kobayashi ~\cite{Ido-Jang-Kobayashi} and Qiu-Zou-Guo~\cite{Qiu-Zou-Guo}.

In contrast, note  that Theorem~\ref{main theorem} determines the bridge surface distance precisely for bridge surface defined by a highly twisted plat projection. This is particularly interesting in the context of the following Theorem of Maggy Tomova:

\begin{theorem*} [Tomova~\cite{Tomova}] 
If $K \subset S^3$ is an $m$-bridge knot with respect to a minimal bridge sphere 
$\Sigma$ such that $d(\Sigma) > 2m$ then $\Sigma$ is the unique minimal bridge sphere.
\end{theorem*}

By Theorem~\ref{main theorem}, a highly twisted knot or link $K \subset S^3$ with $n > 4m(m - 2)$ will have distance $d(\Sigma) > 2m$, so this will define the unique minimal bridge sphere by Tomova's Theorem.

\vskip10pt

It is a result of Schubert~\cite{Schubert} that $2$-bridge knots $K \subset S^3$ are classified by the number $\alpha/\beta \in \mathbb{Q}$ corresponding to the four-strand braid that defines them.  Furthermore, any $4$-plat projection corresponds to a continued fraction expansion of $\alpha/\beta$ and any two such 
projections are equivalent by flype moves (see e.g., Bleiler-Moriah~\cite{BleMo}). Note that flype moves generate twist  boxes with coefficients in $\{- 1, 0, 1 \}$, so these projections are not highly twisted. This  raises a natural question about the possibility of such a classification for  general $m$-bridge knots.

\begin{conjecture}
\label{sphere determines knot} 
If $K \subset S^3$ has a highly twisted $2m$-plat projection of height $n$ such that $n > 4m(m - 2)$ then $K$ has a unique such highly twisted plat projection.
\end{conjecture}

\begin{remark}
\label{classification} 
If  Conjecture \ref{sphere determines knot} is true then Theorem \ref{main theorem} gives a ``normal'' form for highly twisted knots/links and hence is a classification theorem for such links. Clearly not all knots are tightly  twisted but Theorem \ref{main theorem} is a significant step towards a classification.
\end{remark}

The train track argument in this paper is reminiscent of that of Lustig Moriah in~\cite{LuMo} and is organized as follows: In Section~\ref{distance}, we define bridge distance in detail, then in Section~\ref{plats} we give a careful description of plats. We prove the upper bound on the distance of a bridge surface defined by a plat presentation in Section~\ref{upper bound}. 

The majority of the paper is devoted to proving the lower bound. We begin in Section \ref{YYTT} by constructing a collection of ordered train tracks $\{\tau_i\}$ in the bridge surface $\Sigma$. In a sequence of lemmas in Section~\ref{carrying}, the relationships between successive train tracks of this form are described. This determines the intersection pattern between loops carried and almost carried by these train tracks. We then apply these ideas to study pairs of disjoint loops in Section~\ref{disjoint}.

In Section \ref{top and bottom waves} we determine how curves bounding disks in $\mathcal{B}^-$ and $\mathcal{B}^+$ (the complements of the bridge sphere in 
$S^3 \ssm \NN(K)$) are carried by the train tracks. We combine these results to prove Theorem~\ref{main theorem} at the end of this Section.

\vskip10pt

\begin{acknowledgements} 
We thank the Technion, where most of the work was done, for its hospitality. The first author was supported by NSF grant DMS-1308767.
\end{acknowledgements}

\section{Bridge Distance} 
\label{distance}

In this section we present some basic definitions, lemmas and  notions needed for the rest of the paper.

\begin{definition}
\label{curve complex} 
Let $\Sigma_{g, p}$ be a surface of genus $g$ with $p$ punctures. A simple closed curve  $\gamma \subset \Sigma_{g, p}$ is {\it inessential} in $\Sigma_{g, p}$ if it bounds either a disk or a once punctured disk in $\Sigma_{g, p}$. A simple close curve in $\Sigma_{g, p}$ is {\it essential} if it is not inessential. The curve complex  $\mathcal{C}(\Sigma)$ is a simplicial complex defined as follows:
\end{definition}

Let $[\gamma]$ denote the isotopy class of an essential simple closed curve  $\gamma \subset \Sigma$.
\begin{enumerate} 

\item The set of verices of $\mathcal{C}(\Sigma)$ is $\mathcal{V}(\Sigma) = \{ [\gamma]\, | \, \gamma \subset \Sigma$ is essential $\}$.

\vskip5pt

\item An $n$ simplex is an $(n +1)$-tuple $\{[\gamma_0], \dots, [\gamma_n]\}$ of vertices that have pairwise disjoint curve representatives. 

\end{enumerate}

\begin{definition}
\label{bridge distance} 
Suppose $K \subset M$ is a knot in a closed, orientable irreducible $3$-manifold. Let $\Sigma \subset M$ be a sphere decomposing $M$ into balls $V$ and $W$ and assume that $\Sigma$ is transverse to $K$. We will say that $\Sigma$ is a {\it bridge surface} for $K$ if each of the intersections $K \cap V$ and $K \cap W$ is a collection of boundary parallel arcs in $V$ and $W$, respectively. 

Given a bridge surface $\Sigma$ for $K$, define $\Sigma_K = \Sigma \ssm K$, $V_K = V \ssm K$ and $W_K = W \ssm K$. Let $\mathcal{D}(V_K)$ (resp. $\mathcal{D}(W_K)$) be the set of all essential simple closed curves in $\Sigma_K$ that bound disks in $V_K$ (resp. $W_K$).  Define the {\it (bridge) distance} of $\Sigma$ to be $d(\Sigma_K) = d(\mathcal{D}(V_K), \mathcal{D}(W_K))$ measured in $\mathcal{C}(\Sigma_K)$.\end{definition}

\section{Plats}
\label{plats}

In this section, we give a precise definition of plats. While the definition may seem unnecessarily technical, it will prove to be convenient for our purposes. The reader should note that the definition is consistent with the image of Figure \ref{plat}. 

Consider a sweep-out 
$f : S^3 \rightarrow [-\infty,\infty]$ with one index-zero critical point $c_0$ and one index-three critical point $c_3$.  Let $\alpha$ be an arc with endpoints $c_0$ and $c_3$ such that the restriction of $f$ to $\alpha$ is monotonic, and the complement $S^3 \ssm \alpha$ is an open ball $B$. Identify $B$ with $\mathbf{R}^3$, with coordinates $(x,y,z)$ so that each level surface $f^{-1}(t)$ is the plane given by $y = t$.

We will picture the $x$-axis as pointing to the right, the $y$-axis as being vertical and the $z$-axis as pointing towards the viewer. Then the level surfaces $f^{-1}(t)$ appear as horizontal planes.

For each value $y$ and each integer $k$, define $c_{y,k}$ to be the circle in the plane $P_y$ with radius $\frac{1}{2}$, centered at the point $x = k + \frac{1}{2}$, $z = 0$. The {\it plat tube} $A_{i,j}$ is the union of the circles $\{c_{y,2j}\ |\ y \in [i,i+1]\}$ when $i$ is even and the union $\{c_{y,2j+1}\ |\ y \in [i,i+1]\}$ when $i$ is odd. So, the plat tube $A_{i,j}$ is a vertical annulus whose projection onto the $xy$-plane is the square $[k,k + 1] \times [i,i+1]$ where $k = 2j$ for even $i$ and $k = 2j+1$ for odd $i$.

For a pair of integers $\{n, m\}$, the {\it $(n, m)$-plat structure} is the union of the plat tubes $A_{i,j}$ where $i$ ranges from 1 to $n-1$ and $j$ ranges from 1 to either $m$ when $i$ is even, or $m - 1$, when $i$ is odd. Note that the positions of the plat  tubes along the $x$-axis alternate for each row.  Also note that the number of rows is $n - 1$, rather than $n$. This convention will prove more convenient later on.

\begin{definition}
An {\it $(n,m)$-plat braid} is a union of $2m$ pairwise disjoint arcs in $\mathbf{R}^3$ with endpoints in the planes $P_1$ and $P_n$, consisting of arcs contained in an $(n, m)$-plat structure and vertical arcs outside the plat structure. We require that the intersection of the arcs with each plat tube is exactly two properly embedded arcs with endpoints in the plane $z = 0$, and whose projections to the $y$-axis are monotonic. 
\end{definition}

For each plat tube $A_{i,j}$, if we isotope the two arcs within $A_{i,j}$ to intersect the plane $z = 0$ in a minimal number of components, then they will have the same number of components for each of the two arcs in a given plat tube. Note that each arc intersects the plane $z = 0$ at its endpoints, so the number of components of intersection will be 1 if and only if the arc is entirely contained in this plane.

The {\it twist number} $a_{i,j}$ will be the absolute value of this number of components minus one. (If each arc is vertical, i.e.\ contained in the plane $z=0$ then $a_{i,j} = 0$.) The sign of the twist number will be determined as follows: The direction of the $y$-axis defines an orientation on each arc. The projection of each arc into $P_i$ thus defines an orientation on the circle $c_{i,j}$. If this orientation is counter-clockwise, then $a_{i,j}$ will be positive. Otherwise, $a_{i,j}$ will be negative. This translates to the usual ``right hand rule'' for twist regions.

There is a canonical way to construct a link from an $(n , m)$-plat braid: Note that the points $(j, 1, 0)$ for $1 \leq j \leq 2m$ are end points of the plat braid. For  $n$ even there is a unique (up to isotopy) arc in the half of the plane $z = 0$  below the line defined by $y = 1$ connecting the point  $(2j - 1, 1, 0)$ to $(2j, 1, 0)$ for each $1 \leq j \leq m$. When $n$ is even, there are similar arcs for the end points in the half of the plane $z = 0$  above the line defined by $y = n$, and we can choose these arcs to be pairwise disjoint.

When $n$ is odd, we will define the upper arcs by a slightly different construction. In this case, we will attach an arc in the plane $z=0$ from the endpoint $(1, n, 0)$ to $(2m,n,0)$, then for each $2 \leq j \leq m - 1$, we'll add an arc from $(2j - 2 , n, 0)$ to $(2j - 1, n, 0)$. In either case, the union of the plat braid with these upper arcs and the same lower arcs will be called a {\it link in a plat projection} or a {\it plat link}.

The closure in $S^3$ of each plane $P_y$ is a sphere in $S^3$ whose complement is a pair of open balls. By construction, for $1 \leq y \leq n$, $L$ intersects each of these balls in a collection of boundary parallel arcs. Therefore, each $P_y$ for these values of $y$ defines a bridge surface for $L$.

For each integer $i \in [1, n]$ there is a Euclidean projection map $\sigma_i :\mathbf{R}^3 \to P_i$ that sends each point $(x,y,z)$ to $(x,i,z)$. This map sends the plat braid to a union of overlapping circles. For integers $i,j$, $\sigma_i(P_j)$ sends the points of $K \cap P_j$ to the points of $K \cap P_i$, but otherwise the map $\sigma_i$ will not be well behaved relative to the braid. 

In addition to $\sigma$, we will define a second type of projection $\pi$ that will take the braid into account:

Note that a plat braid intersects each $P_y$ in the same number of points and these points vary continuously as $y$ varies from $1$ to $n$.  This can be thought of as an isotopy of these $2m$ points in $P_1$ which extends to an ambient isotopy of $P_1$. Intuitively speaking, the points twist round each other in a manner determined by  the braid, as in the standard analogy between a braid and a path in a configuration space of points.  To be precise, there is a projection  map $\pi_y : \mathbf{R} \times [0,n] \times \mathbf{R} \rightarrow P_y$ for each $y \in [1,n]$ that sends each arc component of the plat braid to a point $(j,y,0)$ and defines a homeomorphism $P_{y'} \rightarrow P_y$ for each $y' \in [0,n]$. These homeomorphisms are canonical up to isotopy fixing the points $K \cap P_y$, and the induced homeomorphism $P_y \rightarrow P_y$ is the identity.

\begin{remark}
\label{The bridge sphere is P1} 
The bridge sphere defined as the closure of $P_1$ will serve as the canonical bridge sphere for the plat link $L$  and will be denoted by $\Sigma$. Any loop $\ell$ in a bridge sphere defined by a different plane $P_{y'}$ can be thought of as a loop in $\Sigma$ by considering its image $\pi_1(\ell)$. 
\end{remark}

\section{Upper bound}
\label{upper bound}

In this Section, we prove the following Lemma, which gives an upper bound on the bridge distance of a tightly twisted knot/link.

\begin{lemma}
\label{upper bounds on d} 
Let $K \subset S^3$ be knot or link with a highly twisted $n$-row, $2m$-plat projection and $\Sigma$ the induced bridge surface. Then $d(\Sigma) \leq  \lceil n /(2(m - 2))  \rceil$.
\end{lemma}

Given a plat link $L$ as above, we will construct a set of ``canonical'' loops and show that this set contains a path of the desired length.

For each odd value of $i$ such that $1 \leq i \leq n$, and each value $1 \leq j \leq m-1$, let $\ell_{i,j}$ be the circle in the plane $P_i$ with radius $\frac{3}{4}$, centered at $(2j + \frac{3}{2}, i, 0)$. Note that this circle bounds a disk $E_{i,j}$ in $P_i$ containing the circle $c_{i,j}$ defined above, as well as the two points of $L \cap P_i$ that are in $c_{i,j}$. These loops are shown in the odd rows of Figure~\ref{loops}(a).

Similarly, for each even value of $i$ such that $1 \leq i \leq n - 1$, and each value $1 \leq j \leq m$, let $\ell_{i,j}$ be the circle in the plane $P_i$ with radius $\frac{3}{4}$, centered at $(2j + \frac{1}{2}, i, 0)$. Again, this circle bounds a disk $E_{i,j}$ in $P_i$ containing the circle $c_{i,j}$ and the two points of $L \cap P_i$ that are in $c_{i,j}$. These loops are shown in the even rows of Figure~\ref{loops}(a).

For $1 \leq j \leq m$, define $\ell_{0, j}$ to be the circle in $P_1$ of radius $\frac{3}{4}$ centered at $(2j - \frac{1}{2}, 1, 0)$. Note that these loops intersect the loops $\ell_{1,j}$, while the rest of the loops defined so far are pairwise disjoint. (We could modify the construction to make them disjoint, but as the reader will find below, this wouldn't be worth the trouble.) Regardless, by construction, loop $\ell_{0,j}$ is the boundary of a disk $D^-_1$ in $R^3$ disjoint from $L$ whose interior is contained in the half space below $P_1$, i.e.\ a compressing disk for the bridge surface defined by $P_1$, as in Figure~\ref{loops}(a).

Similarly, for the final row $i = n$, the loops $\ell_{n,j}$ will bound compressing disks with interiors above $P_n$.
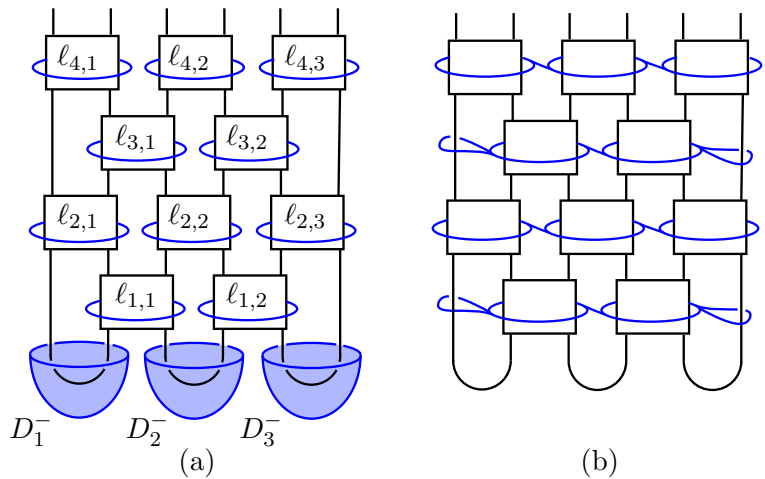
\begin{figure}[ht]
  \includegraphics[width=4in]{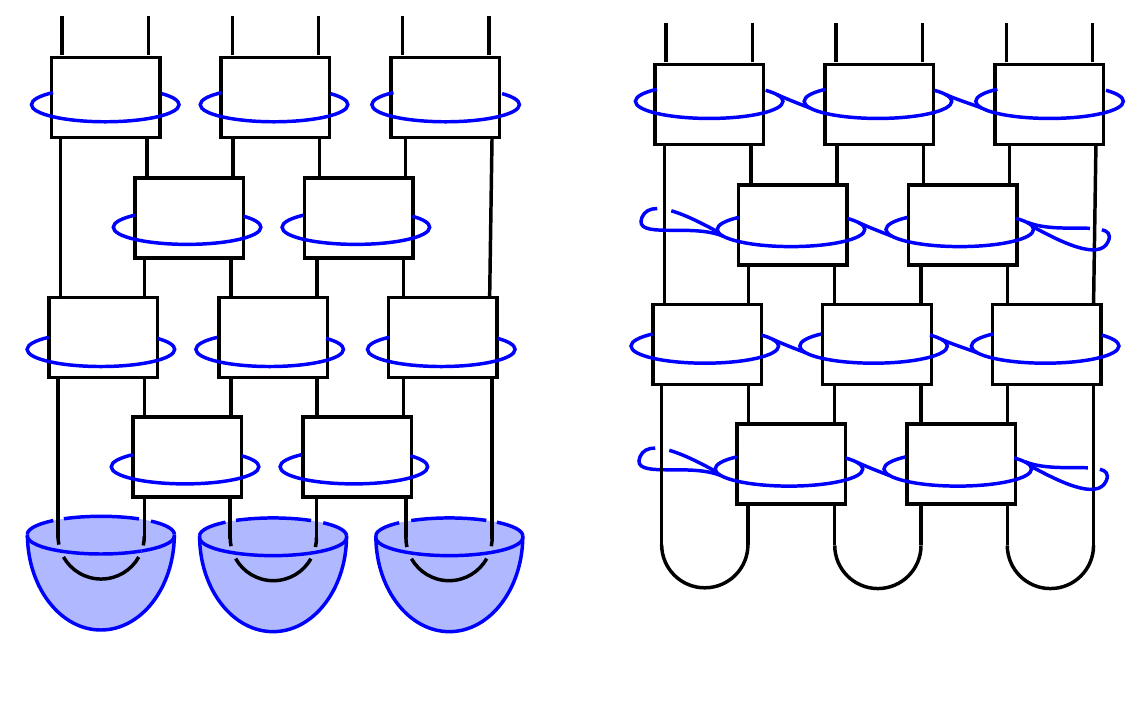}
  \put(-250,63){$\ell_{1,1}$}
  \put(-208,63){$\ell_{1,2}$}
  \put(-272,93){$\ell_{2,1}$}
  \put(-230,93){$\ell_{2,2}$}
  \put(-186,93){$\ell_{2,3}$}
  \put(-250,123){$\ell_{3,1}$}
  \put(-208,123){$\ell_{3,2}$}
  \put(-272,153){$\ell_{4,1}$}
  \put(-230,153){$\ell_{4,2}$}
  \put(-186,153){$\ell_{4,3}$}
  \put(-290,12){$D^-_1$}
  \put(-246,12){$D^-_2$}
  \put(-203,12){$D^-_3$}
  \put(-226,0){(a)}
  \put(-74,0){(b)}
  \caption{The loops $\ell_{i,j}$ and train track diagrams $\tau_i$}
  \label{loops}
\end{figure}
 
\begin{lemma}
\label{big step}
For every even value $i < n$ and every integer $k \leq 2(m-2)$ such that $i + k \leq n$, $\pi_1(\ell_{i,m})$ is disjoint from the left-most loop $\pi_1(\ell_{i+k, 1})$ and $\pi_1(\ell_{i,1})$ is disjoint from the right-most loop $\pi_1(\ell_{i+k, a})$ in row $i+k$, where $a = m$ when $i + k$ is even and $a = m-1$ when $i + k$ is odd.
\end{lemma}

\begin{proof} The proof is an immediate consequence of the definitions of the functions $\pi_i$.
\end{proof}

\begin{proof}[Proof of Lemma~\ref{upper bounds on d}]
Let $r$ be the largest integer such that $r(2(m - 2)) \leq n$. By repeatedly applying Lemma~\ref{big step}, we see that $\ell_{0,1}$ is distance at most $r$ from either $\ell_{2r(m-2),m}$ (when $r$ is odd) or $\ell_{2r(m-2),1}$ (when $r$ is even). If $r(2(m - 2)) = n$ then both loops $\ell_{2r(m-2),m}$ and $\ell_{2r(m-2),1}$ bound disks above $\Sigma$, so $d(\Sigma) \leq r = n/(2(m-2)) = \lceil n/(2(m-2)) \rceil$.

Otherwise, since $n - 2r(m-2) < 2(m-2)$, Lemma~\ref{big step} implies that each of $\pi_1(\ell_{2r(m-2), 1})$ and $\pi_1(\ell_{2r(m-2), m})$ is disjoint from some loop $\pi_1(\ell_{n,a})$ in row $n$. By construction each of these loops bounds a disk above $P_n$. Thus $d(\Sigma) \leq r + 1$. Since $r(2(m-2)) < n < (r+1)(2(m-2))$, we have $\lceil n/(2(m-2)) \rceil = r+1 \geq d(\Sigma)$, completing the proof.
\end{proof}

\begin{remark}\label{n- bounds}  
In particular, this implies that when $n < 2(m-2)$, the bridge distance of $\Sigma$ is at most one. Thus from now on, unless specifically noted otherwise we will assume that $n \geq 2(m - 2)$.
\end{remark}

\section{Taos and Train tracks}
\label{YYTT}

In this section, we define train tracks. For reasons that will become clear below, our definition is slightly different from the usual one.

\begin{definition}\label{train track} 
A {\it train track} $\tau$ is a compact subsurface of $\Sigma$ with a singular fibration by intervals: The interior of $\tau$ is fibered by open intervals and the fibration extends to a fibration of the surface with boundary by properly embedded closed intervals except for finitely many intervals called {\it singular fibers}. Each singular fiber $\alpha$ has a neighborhood in $\tau$ homeomorphic to $([0,1]\times[0,1]) \ssm((\frac{1}{4}, \frac{3}{4}) \times (\frac{1}{2}, 1))$ such that $\alpha$ is the horizontal interval $[0,1] \times \{\frac{1}{2}\}$ and the adjacent fibers are also horizontal, as in Figure~\ref{ttlocal}.
\end{definition} 

Given a train track $\tau$, the projection of every fiber to a point results in a graph $T$ called the {\it train track diagram} for $\tau$. This graph has a natural embedding in the surface $\Sigma$, though not in the train track $\tau$. Each singular fiber in the train track determines a trivalent vertex of $T$, called a \textit{switch}. The interval bundle structure on the original train track determines a vector at each switch to which each adjacent edge is tangent. Thus a train track diagram is a graph such that the three edges adjacent to each given vertex are tangent to the same vector, though they will approach the vertex from both sides along that vector, as on the right in Figure~\ref{ttlocal}. 
\begin{figure}[ht]
\includegraphics[width= 4.5in]{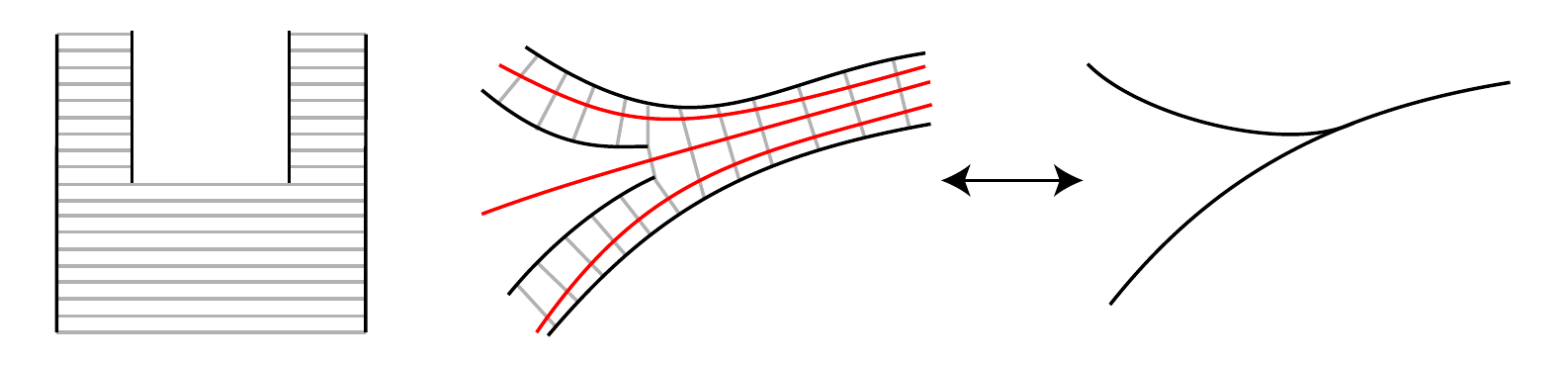}
\caption{Converting between a train track and a train track diagram}
\label{ttlocal}
\end{figure}

Conversely, given a graph $T$ in $\Sigma$ that satisfies this property concerning tangent vectors at vertices, one can construct a train track $\tau$ with train track diagram $T$. Note that we are restricting our attention to train tracks whose diagrams have three-valent vertices, though this is not necessary in general.

\vskip10pt

In Section~\ref{upper bound}, we constructed a collection of ``canonical'' loops that contained a path giving an upper bound on the distance of $\Sigma$. To find the lower bound, we will show that, roughly speaking, any other path must follow very close to one of these canonical paths. To do this, we will extend the canonical loops to a sequence of train tracks  $\tau_i$ defined by train track graphs $T_i$ that encode the structure of the plat link $K$.

In each disk $E_{i,j} \subset P_i$ bounded by the loop $\ell_{i,j}$, there is a unique (up to isotopy) properly embedded arc $\alpha$ that separates the two punctures. We can isotope this arc to intersect the plane $z=0$ in a single point. We can further isotope this arc in a neighborhood of its boundary so that it is tangent to $\ell_{i,j}$, and there will be two choices for how to do this. 

To differentiate the two, recall that each $P_y$ is oriented, which defines an orientation on each loop $\ell_{i,j}$. For the figures below, we will draw each $\ell_{i,j}$ so that the induced orientation is counter-clockwise.

Consider a tangent vector at one of the endpoints of $\alpha$, pointing out of the arc. When we make $\alpha$ tangent to $\ell_{i,j}$, this tangent vector will point in a direction that will either agree with the orientation on $\ell_{i,j}$ or disagree. If we isotope $\ell$ so that both tangent vectors agree with the orientation on $\ell_{i,j}$ then we will say that the image of this isotopy is a {\it right handed tao arc}. Otherwise, if both tangent vectors disagree, we will call it a {\it left handed tao arc}.

\begin{definition}\label{The train tracks}
The union of each loop $\ell_{i,j}$ and a (left handed/right handed) tao arc will be called a {\it (left handed/right handed) tao diagram}, as in Figure~\ref{yinyangs}. (The name alludes to the Taoist ``yin-yang'' symbol, which the figure resembles.)
\end{definition}

\begin{figure}[ht]
\includegraphics[width= 2in]{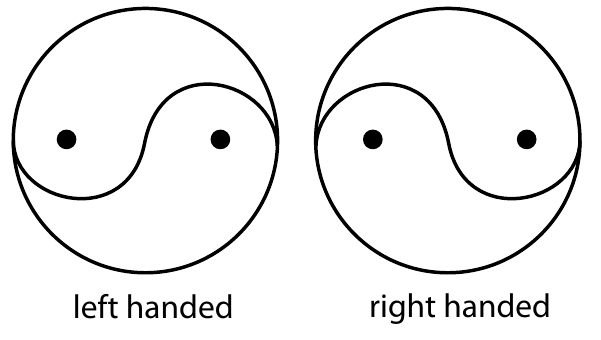}
\caption{Left handed and right handed tao diagrams.}
\label{yinyangs}
\end{figure}

For each $\ell_{i,j}$, let $Y_{i,j}$ be a tao diagram, such that $Y_{i,j}$ is right handed when $a_{i,j}$ is positive and left handed otherwise. (In Theorem \ref{main theorem}, we assume that $|a_{i,j}| \geq 3$, so it won't matter which handedness we pick in the case when $a_{i,j} = 0$.) 

For each $i$ and the appropriate values of $j$, there will be an arc of the line $\{(x,i,0)\}$ connecting $Y_{i,j}$ to $Y_{i,j+1}$. As with the tao arcs, we can isotope the ends of this arc, within $P_i$, to make it tangent to $\ell_{i,j}$ and $\ell_{i,j+1}$. We will say that the resulting arc is {\it compatible} with the adjacent taos if the direction of a tangent vector pointing into the arc agrees with orientation defined by the tao arc on each loop. Moreover, we will isotope each endpoint of the connecting arc in the direction of an out-pointing tangent vector along $\ell_{i,j}$ past exactly one endpoint of the tao arc.

For each even value of $i$, we will define $T_i$ to be the union of the tao diagrams $\{Y_{i,j}\}$ and a collection of compatible arcs. Some of the possible diagrams for $m=3$ (depending on the signs of the coefficients $a_{i,j}$) are shown on the left of Figure~\ref{ttdiagrams}. The way these train tracks sit with respect to the plat is shown in Figure~\ref{loops}(b).

\begin{figure}[ht]
\includegraphics[width= 4in]{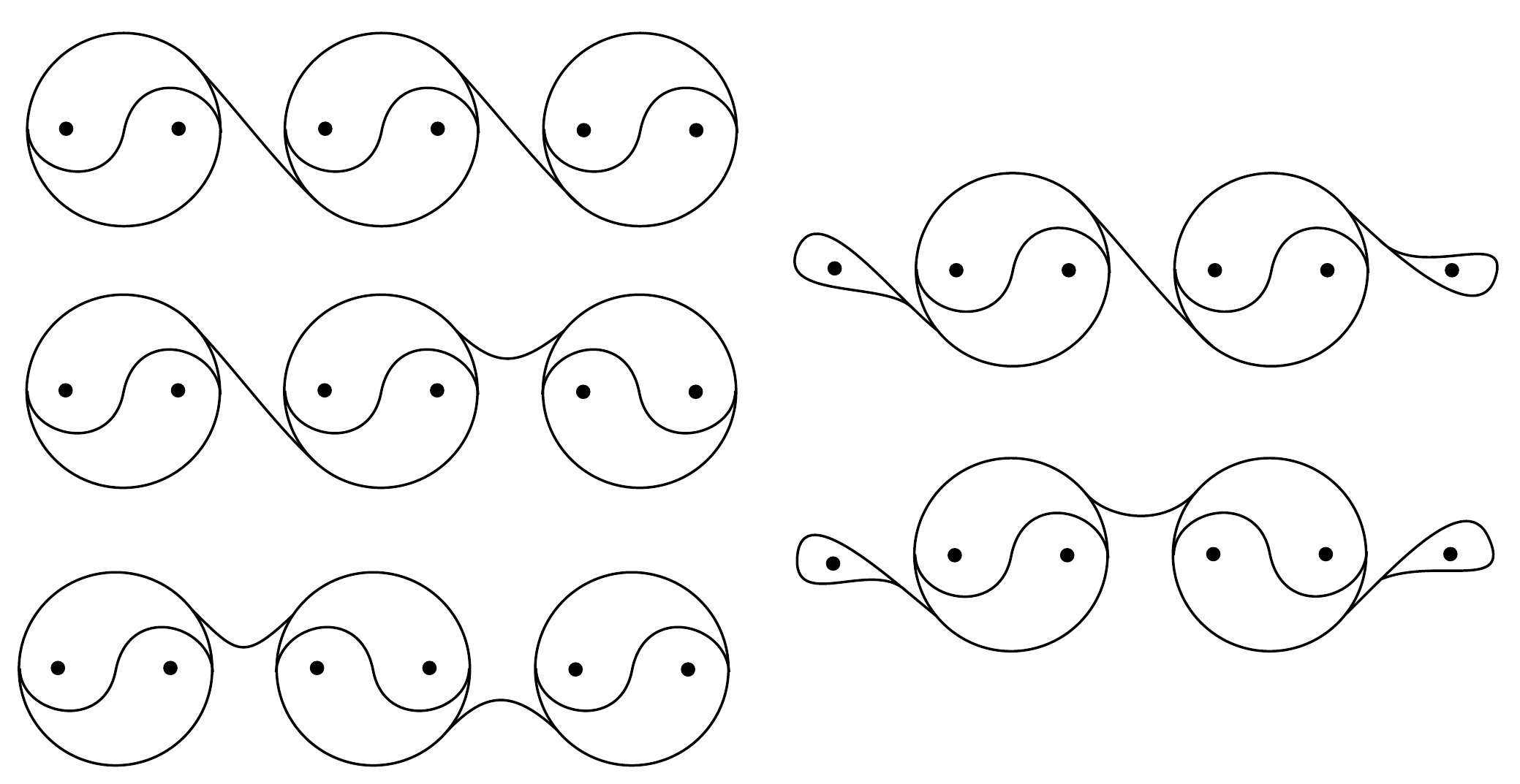}
\caption{The train track diagrams for the $\tau_i$, as in Definition \ref{The train tracks}. The diagrams on the left are for $i$ even 
and those on the right are for $i$ odd.}
\label{ttdiagrams}
\end{figure}

For odd values of $i$, there are two points of $\ell \cap P_i$ that are not contained within any of the tao diagrams, and we will need to make an extra consideration for these. For each such value of $i$, there is an arc in $P_i$ with endpoints in $Y_{i,1}$ that intersects the line $z = 0$ in a single point and bounds a disk containing the point $(1,i,0)$ (but not the point $(2m,i,0)$). Isotope the endpoints of this arc to the same point in $Y_{i,1}$, then make them tangent to $\ell_{i,1}$ so that each of the two vectors pointing into the arc agree with the orientation of $\ell_{i,1}$. Finally, pinch the ends of the two arcs together, so that there is a single arc from $Y_{i,1}$ to a switch that splits the two ends of the original arc, as on the right part of Figure~\ref{ttdiagrams}. The resulting switch and two arcs will be called an {\it eyelet}, and we will include a second eyelet around the point $(2m, i, 0)$. 

For each $i$, define $\tau_i$ to be the train track defined by the train track diagram $T_i$. Note that $\tau_i$ is connected and is made up of train tracks induced by the tao diagrams, as in Figure~\ref{yinyangstt}, connected in pairs by bands. We will call each train track constructed in this way a {\it plat track}. These are shown relative to the entire plat in Figure~\ref{loops}.

\begin{figure}[ht]
\includegraphics[width= 3.5in]{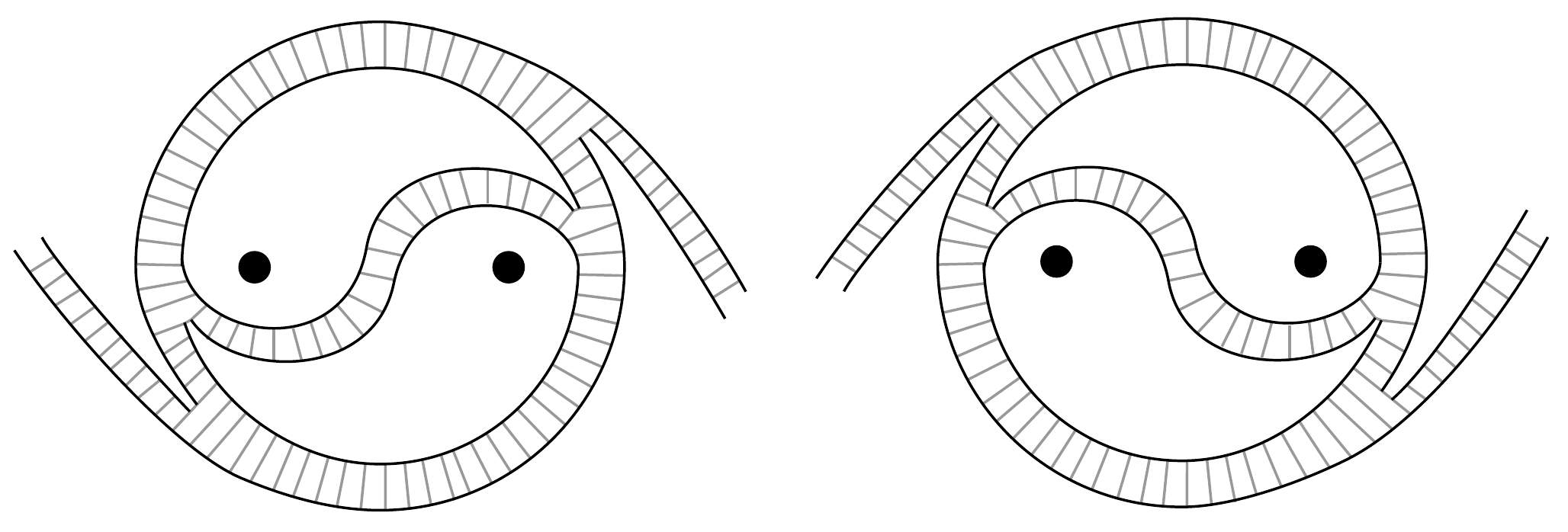}
\caption{Left handed and right handed taos.}
\label{yinyangstt}
\end{figure}

\section{Carrying and covering}
\label{carrying}

A train track $\tau$ is said to {\it carry} a simple closed curve $\ell$ if $\ell$ is contained in (the subsurface) $\tau$ and $\ell$ is transverse to each fiber in $\tau$ that it meets. Moreover, $\tau$ carries a given isotopy class of loops if it carries at least one loop in the isotopy class. 

In this paper, we will need a slightly more general verion of the notion of carrying. We will say that a properly embedded arc $\gamma \subset \tau$ is {\it carried} by $\tau$ if its interior is transverse to the interval fibers of $\tau$ and each endpoint is in the intersection of the interior of a singular fiber with the boundary of $\tau$. 

\begin{definition}
A train track {\it almost carries} a loop $\ell$ if 
\begin{enumerate}
\item $\ell$ is disjoint from the endpoints of the interval fibers in $\tau$,
\vskip4pt
\item its intersection with $\tau$ is carried by $\tau$,
\vskip4pt 
\item no arc of $\ell$ in the complement of $\tau$ is parallel into the interior of a singular fiber and
\vskip4pt
\item no arc of $\ell$ in the complement of $\tau$ is parallel into an arc of fiber endpoints.
\end{enumerate}

For the plat tracks $\tau_{i,j}$ with $i$ odd, we need to amend this definition slightly: We will also say that a loop $\ell$ is {\it almost carried} by a plat track $\tau_i$ if it is almost carried by the train track defined by the track diagram that results from removing exactly one of the two eyelets from $T_i$. 

\end{definition}

The types of arcs ruled out by conditions (3) and (4) are shown in Figure~\ref{cbadarcs}.

\begin{figure}[ht]
  \includegraphics[width= 3.5in]{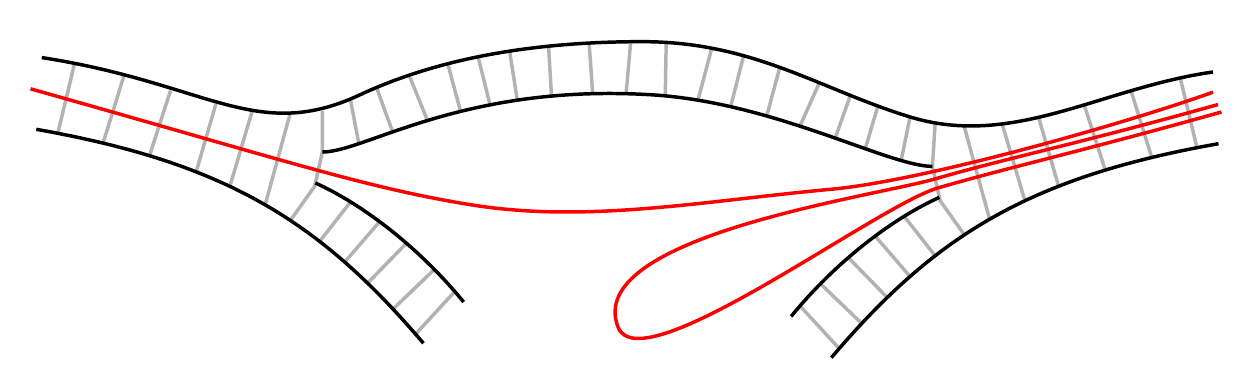}
    \caption{Trivial arcs that cannot appear in almost carried loops.}
  \label{badarcs}
\end{figure}

\begin{remark}\label{enter and leave}
Note that the boundary of a train track $\tau$ is a union of the endpoints of the interval fibers and arcs in the interiors of the singular fibers that define the switches. Thus a loop that is almost carried by $\tau$ can only enter and leave the train track $\tau$ through the switches, as indicated by the arcs in Figure~\ref{ttlocal}.
\end{remark}

\begin{definition}\label{covers} 
A loop $\ell$ that is carried or almost carried by a train track $\tau$ is said to {\it cover} $\tau$ if $\ell$ meets every fiber of $\tau$.
\end{definition} 

We will also need a slightly more general version of this definition that is specific to the plat tracks defined above. As noted above, there is a canonical map from a train track $\tau$ to its diagram $T$ that collapses each interval fiber to a point. Given a subgraph $G$ of a train track diagram $T$ and a loop $\ell$ that is almost carried by $\tau$, we will say that $\ell$ {\it covers} $G$ if the set of carried arcs is sent onto $G$ by the map $\tau \rightarrow T$. In other words, the set of carried arcs must intersect every arc that defines a point in the subgraph.

Given a positive integer $k$, we will say that a loop $\ell$ that is almost carried by a train track $\tau_i$ {\it covers $k$ taos} if $\ell$ covers a subgraph of $T_i$ containing $k$ distinct tao diagrams. Similarly, we will say that $\ell$ {\it covers $k$ taos and an eyelet} if it covers a subgraph of $T_i$ containing $k$ tao diagrams and an eyelet subgraph.

The following lemmas only need to assume that the coefficients in a single row are sufficiently large. In order to prove the most general statement, we will introduce one more definition: We will say that row $i$ of a plat is {\it highly twisted} if $|a_{i,j}| \geq 3$ for each appropriate $j$. So a plat will be highly twisted if every one of its rows is highly twisted.

\begin{lemma}
\label{carried and covers}
If row $i$ of the plat defining $K$ is highly twisted then $\pi_i(\ell_{i+1,j})$ is carried by $\tau_i$ and covers either two taos or one tao and one eyelet.
\end{lemma}

\begin{proof}
Let $\ell'_{i+1,j}$ be the Euclidean projection $\sigma_i(\ell_{i+1,j})$ in the horizontal plane $P_i$. When $i$ is odd and
$j = 0$ or $j = m$, this loop will intersect one of the taos in $\tau_i$ transversely. Otherwise, it will intersect two taos transversely. In particular, for each of these one or two tao diagrams $Y_{i,k}$ in $\tau_i$, there is an arc $\alpha$ of $\ell'_{i+1,j}$ that intersects $Y_{i,k}$ in three points, passing between the two points of the  knot $K$ inside the disk $E_{i,k}$, as on the left side of Figure~\ref{taospin}. To get $\pi_i(\ell_{i+1,j})$ from $\ell'_{i+1,j}$ one needs to apply three or more, half twists along the loop $\ell_{i,k}$ for the relevant values of $k$. The reader can check that because $|a_{i,k}| \geq 3$ and each tao diagram has the correct handedness, the image of $\alpha$ under this twist consists of a subarc carried by $\tau_i$ and two subarcs disjoint from $\tau_i$, as on the right in Figure~\ref{taospin}. Since the rest of $\ell'_i$ is disjoint from $\tau_i$, the loop $\pi_i(\ell_{i+1,j})$ is carried by $\tau_i$, after perhaps pushing the endpoints of the connecting arcs into $\tau_i$.
\end{proof}

\begin{figure}[ht]
\includegraphics[width= 4.5in]{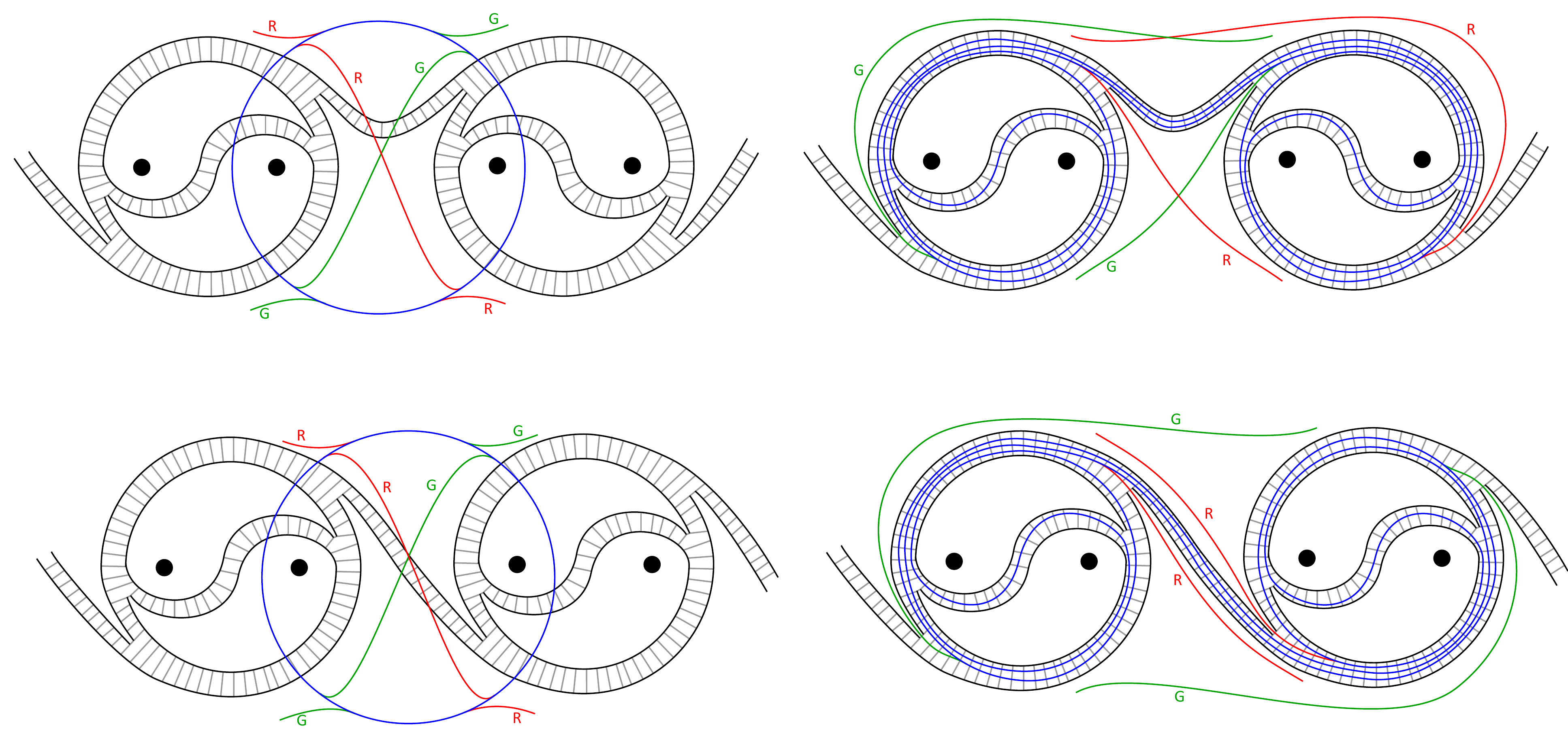}
\caption{The arc $\alpha$ and its image after twisting around $\ell_i$ by three half twists.}
\label{taospin}
\end{figure}

\begin{lemma}
\label{coversdisjointcarried}
Let $\tau \subset \Sigma$ be a train track and $\ell \subset \Sigma$ a loop that is almost carried by $\tau$ and covers $\tau$. If $\ell' \subset \Sigma$ is a simple closed curve disjoint from $\ell$ then there is an isotopy after which $\ell'$ is almost carried by $\tau$.
\end{lemma}

\begin{proof}
Let $\NN = \NN(\ell) \subset \Sigma$ be a regular neighborhood of $\ell$ disjoint from $\ell'$. Because $\ell$ covers $\tau$, every interval fiber intersects $\NN$ in one or more subintervals. There is thus an isotopy of $\tau$ that takes each fiber endpoint into $N$, so that the complement $\tau \ssm \NN$ is a union of intervals in the interiors of the interval fibers. These intervals are properly embedded in $\Sigma \ssm \NN$ and make up parallel families that form bands. 

Assume we have isotoped $\ell'$ so as to minimize its intersection with each of these bands. Then each arc of 
$\ell' \cap (\tau\ssm \NN$) will be essential in a band and can be made transverse to each interval fiber. If any arc of 
$\ell' \ssm \tau$ is isotopic into the interior of a singular fiber then it can be pulled into one of the bands, then across to 
reduce the number of arcs of intersection. 

Thus $\ell'$ can be isotoped so that it satisfies the first three conditions in the definition of almost carried. For the fourth condition, note that if an arc of $\ell' \ssm \tau$ is parallel to an arc composed of interval endpoints then this arc can be isotoped into $\tau$ to an arc in $\tau$ transverse to all the fibers. If we do this for each such arc, then $\ell'$ will be almost carried by $\tau$.
\end{proof}

\begin{definition}
\label{ttcarried}
We will say the a train track $\tau'$ is {\it carried} by a train track $\tau$ if $\tau' \subset \tau$ and each interval fiber of $\tau'$ is contained in an interval fiber of $\tau$. This is equivalent (up to isotopy) to the statement that each edge in the train track diagram for $\tau'$ is transverse to the interval fibers of $\tau$. 

We will say that, a train track $\tau'$ is {\it almost carried} by a train track $\tau$ if the train track diagram $T'$ for $\tau'$ is transverse to the interval fibers, disjoint from their endpoints, and no arc of $T'$ in the complement of $\tau$ is parallel into the interior of a singular interval or an arc of endpoints. 
\end{definition}

Being almost carried is equivalent to the condition that each interval fiber of $\tau'$ is either disjoint from $\tau$ or contained in an interval fiber of $\tau$, plus a condition on the bands of $\tau'$ outside of $\tau$.

\begin{lemma}
\label{LLCarried} 
If row $i$ of the plat defining $K$ is highly twisted then $\pi_i(\tau_{i + 1})$ is almost carried by $\tau_i$ and the union of the loops $\pi_i(\ell_{i + 1, j})$ covers $\tau_i$.
\end{lemma}

\begin{proof} 
It follows from Lemma~\ref{carried and covers} that the loops $\pi_i(\ell_{i+1,j})$ cover $\tau_i$. The rest of the train 
track $\tau_{i+1}$ is disjoint from these loops, so it might be expected that the result will follow from 
Lemma~\ref{coversdisjointcarried}. In fact, in order to use an argument very similar to the proof of Lemma~\ref{coversdisjointcarried}, it must first be checked that the switches in $\pi_i(\tau_{i+1})$ are compatible with  the train track $\tau_i$. 

The arcs of $T_{i+1}$ near two of the loops $\ell_{i+1,j}$ are shown on the left in Figure~\ref{taospin}, in either red or  green. The green arcs are for a left handed tao and the red arcs are for a right handed tao. For anyone reading this in grey scale, each red or green arc has been labeled with an `R' or `G', respectively. The image of each tao arc under $\pi_i$ is a short arc connecting the tracks between the adjacent taos in $\tau_i$. (The images of the tao arcs are not shown in the lower half of Figure~\ref{taospin}  because  they would be too small to see on this scale.) 

The images of the connecting arcs, unfortunately, will intersect the endpoints of the interval fibers in $\tau_i$. However, note that for a right handed $Y_{i+1,j}$, one can isotope the endpoints of the connecting arcs in a clockwise direction along a right handed $\ell_{i+1,j}$, without crossing the endpoints of the tao arc. Similarly, we can slide the the endpoints of the connecting arcs for a left handed $\ell_{i+1,j}$ in a counter-clockwise direction. Therefore,  the endpoints can be pushed along $\ell_j$ to an arc that exits $\tau_i$ at the interior of a switch, as on the right side in Figure~\ref{taospin}. Now apply the same argument as in the proof of Lemma~\ref{coversdisjointcarried}, using the fact that $\pi_i(\ell_{i+1})$ covers $\tau_i$, to isotope the interiors of the connecting arcs to be almost carried by $\tau_i$.
\end{proof}

The main objective of the above lemmas is the following corollary, which is immediate:

\begin{corollary}
\label{projection coro}
If $\ell$ is almost carried by $\tau_{i+1}$ then $\pi_i(\ell)$ is almost carried by $\tau_i$.\qed
\end{corollary}

\section{Stepping through the plat}
\label{disjoint}

\begin{definition}
We will say that a loop $\ell$ that is almost carried by one of the train tracks $\tau_i$ {\it bisects a tao} if it either covers one of the loops $\ell_{i,j}$ or covers one of the tao arcs in $\tau_i$. We will say that $\ell_i$ {\it bisects $t$ taos} if it covers at least $t$ of the tao arcs in $\tau_i$. We will say that $\ell$ {\it bisects $t$ taos and an eyelet} if it covers at least $t$ tao arcs and one of the eyelets in $\tau_i$.
\end{definition}

\begin{lemma}
\label{hitsonetao} 
If $\ell$ is an essential loop that is almost carried by $\tau_i$ then $\ell$ bisects a tao.
\end{lemma}

\begin{proof}
Consider the bridge sphere $\Sigma_i$ defined by the plane $P_i$. The complement in $\Sigma_i$ of $\tau_i$ is a collection of disks, each of which has at most one puncture, every essential loop in $P_i$ must intersect $\tau_i$. If $\ell$ is almost carried by $\tau_i$ then $\ell \cap \tau_i$ is either a loop or a collection of arcs that are carried by $\tau_i$. A connecting arc between the taos cannot fully carry an arc on its own, so every carried arc or loop must cover one of the edges of a tao. Moreover, the reader can check that any arc or loop that covers one edge of a tao must either cover the tao arc or be a loop that follows the outer loop. Thus, since $\ell$ covers an arc of some $Y_{i,j}$, it must bisect a tao. 
\end{proof}

\begin{lemma}
\label{up the ladder} 
Suppose that $K$ is defined by a plat in which row $i$ is highly twisted and $\ell$ is a loop that is almost carried by $\tau_{i+1}$. If $\ell$ bisects $t$ taos in $\tau_{i+1}$ then $\pi_i(\ell)$ either bisects $t +1$ taos in $\tau_i$ or bisects $t$ taos and an eyelet. 
\end{lemma}

\begin{proof}
By Corollary~\ref{projection coro}, $\pi_i(\ell)$ will be almost carried by $\tau_i$ since $\ell$ is almost carried by $\tau_{i + 1}$. If $\ell$ is isotopic to the loop defining a tao then $k = 1$ and by Lemma~\ref{carried and covers}, $\ell$ covers two taos or one tao and one eyelet in $\tau_i$. 

Otherwise, $\ell$ must cover $t$ tao arcs in $\tau_{i+1}$. Any arc $\alpha$ carried by a tao arc must also cover the two arcs of the outer circle into which the tangent vectors, coming out of the tao arc, point. Such arcs are shown on the left in Figure~\ref{taospinarcs}, superimposed on $\tau_i$.

\begin{figure}[ht]
\includegraphics[width= 4.5in]{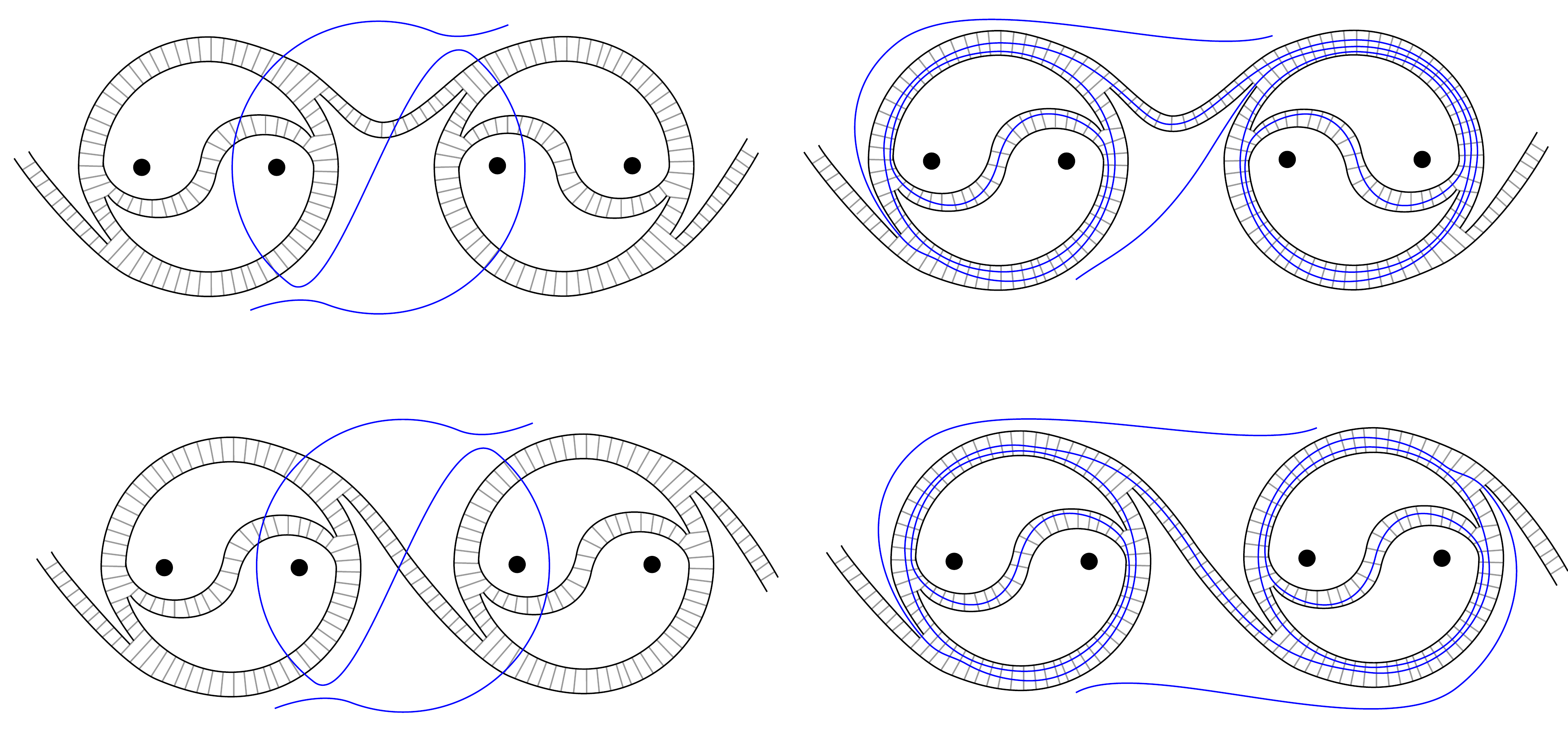}
\caption{The arc $\alpha$ and its image after twisting around $\ell_i$ by three half twists.}
\label{taospinarcs}
\end{figure}

	The arc $\alpha$ intersects either two taos or one tao and one eyelet. The reader can check that, as indicated in  Figure~\ref{taospinarcs}, the image $\pi_i(\alpha)$ will cover these two taos or one tao and one eyelet. If $\ell$ covers $k$ adjacent tao arcs in $\tau_{i+1}$ then, after taking into account the overlaps between them, we find that their images will cover either $t + 1$ taos in $\tau_i$ or $t$ taos and one eyelet.
\end{proof}

\section{Transverse loops} 
\label{top and bottom waves}

Lemma~\ref{up the ladder} gives us a good idea of how later loops in a path in $\mathcal{C}(\Sigma)$ behave with respect to the train tracks, as long as we know that the initial loop is almost carried. When we consider paths between any two loops in the disk sets in $\mathcal{C}(\Sigma)$ of a bridge surface $\Sigma$, we can't assume that our initial loop is carried by a train track. However, we can still gain a reasonable amount of control over the first loop if we broaden our ideas of how a loop should be allowed to intersect a train track.

\begin{definition}
\label{def:transverse}
A loop $\ell \subset \Sigma$ is said to be {\it transverse to a train track} $\tau$ if:
\vskip4pt
\begin{enumerate}
\item Every component of $\ell \cap \tau$ is either carried by $\tau$ (i.e.\ transverse to the interval fibers and disjoint from their endpoints) or an arc contained in an interval fiber,
\vskip4pt
\item no arc of $\ell$ in the complement of $\tau$ is parallel into the interior of a singular arc defining an open switch or into an arc of fiber endpoints and
\vskip4pt
\item no component of $\ell \ssm \tau$ is an arc that is parallel into an arc of $\partial \tau$ that intersects zero or one singular fibers.
\end{enumerate}
\end{definition}

Note that almost carried loops are by definition transverse. However, there are two important additions in the definition of transverse loops: First, a transverse loop $\ell$ is allowed to cut across the train track parallel to the fibers. Such a loop is thus allowed to intersect the endpoints of the interval fibers, allowing for a wider variety of behavior in the complement of $\tau$. As a result, we have to rule out the types of arcs labeled (1), (2) and (3) in Figure~\ref{morebadarcs}. As we will see (and as suggested by Figure~\ref{morebadarcs}), if any of these types of arcs occur, we can isotope $\ell$ to reduce its intersection with $\tau$.

\begin{figure}[ht]
\includegraphics[width= 3.5in]{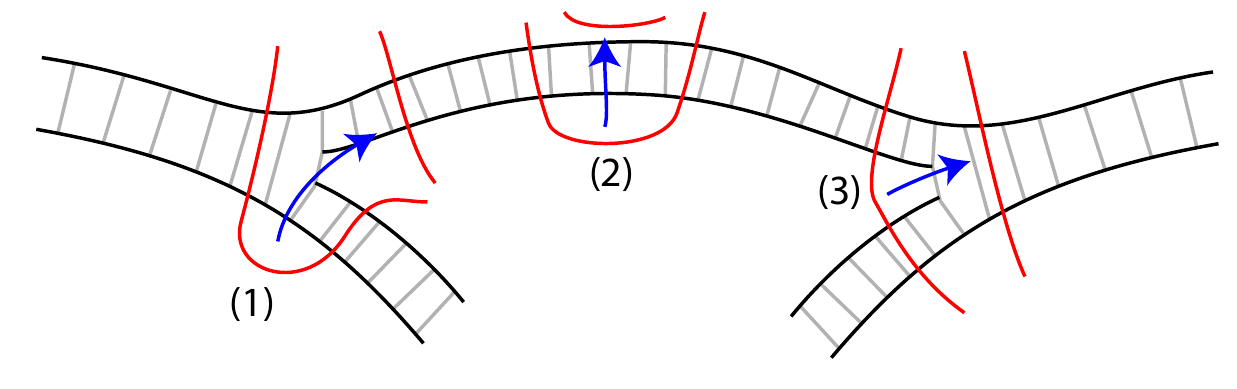}
\caption{Arcs that are ruled out by the third condition on transverse loops, and isotopies that remove them.}
\label{morebadarcs}
\end{figure}

\begin{lemma}
\label{lem:maketransverse}
Given a train track $\tau \subset \Sigma$, every loop $\ell$ in $\Sigma$ is isotopic to a loop that is transverse to $\tau$.
\end{lemma}

\begin{proof}
Let $T$ be a diagram for the train track $\tau$. Isotope $\ell$ to be transverse to $T$, i.e.\ so that $\ell$ is disjoint from the vertices of $T$ and intersects each edge in a finite number of transverse points. Because $T$ is a trivalent graph, if any arc $\alpha$ in $\ell \ssm T$ is parallel to an arc $\beta$ in $T$ that intersects one or zero vertices of $T$ then we can isotope $\alpha$ onto and across $\beta$, reducing the number of points in $\ell \cap T$ by one. Thus if we isotope $\ell$ so that $\ell \cap T$ is minimal then there will be no such $\alpha$, $\beta$ arcs .
 
Isotope $\tau$ onto a small regular neighborhood of $T$ so that the intersection $\ell \cap \tau$ will be a finite number of arcs, each corresponding to a point of $\ell \cap T$ and each parallel to some interval fiber of $\tau$. We can isotope $\ell$ further so that each intersection arcs is in fact an interval fiber. This loop 
$\ell$ satisfies the first condition of Definition~\ref{def:transverse}. Because the arcs of $\ell \ssm \tau$ do not have endpoints in the interiors of singular fibers of $\tau$, the second condition is vacuously satisfied. Finally, if any arc $\alpha'$ of $\ell \ssm \tau$ were parallel to an arc of $\partial \tau$ that intersects one or fewer singular fibers, then we could extend $\alpha'$ to an arc $\alpha$ of $\ell \ssm T$ of the form that was ruled out by the minimality of $\ell \cap T$. Thus we conclude that $\ell$ is transverse to $\tau$.
\end{proof}

Note that the construction in the proof of Lemma~\ref{lem:maketransverse} always produces a loop in which none of the arcs of intersection are carried by $\tau$. In some cases there will be a further isotopy, as indicated in Figure~\ref{trvtocarried}, that replaces a portion of $\ell$ with an arc that is carried by $\tau$. The initial position of $\ell$ is allowed by Definition~\ref{def:transverse} because the arc of $\partial \tau$ that is parallel to the arc of $\ell$ intersects the endpoints of multiple (in this case, four) singular fibers.

\begin{figure}[ht]
\includegraphics[width= 4.5in]{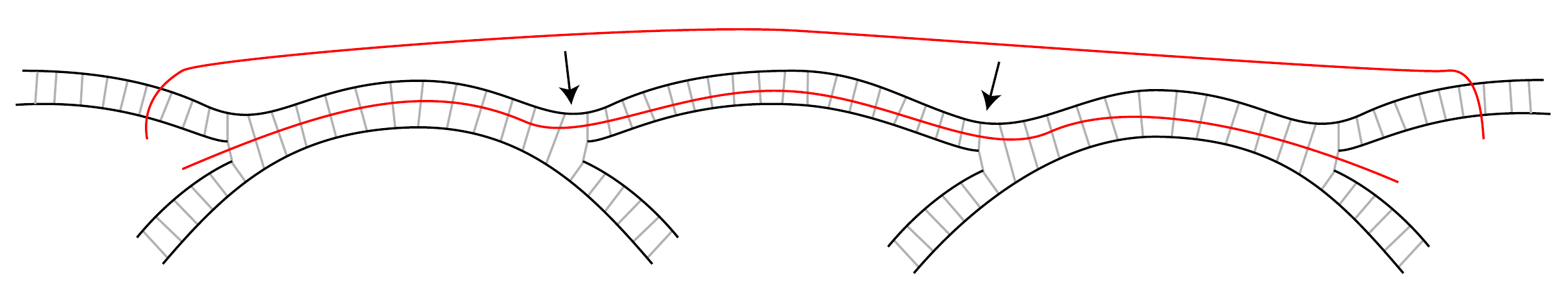}
\caption{Isotoping a transverse arc to a carried arc.}
\label{trvtocarried}
\end{figure}

\begin{definition}
We will say that a transverse loop $\ell$ \textit{covers} a subgraph $G$ of a train track graph $\tau$ if the set of carried arcs covers $G$.
\end{definition}

\begin{lemma}
\label{diskcoversone}
Let $D$ be a properly embedded, essential disk above the bridge surface $P_n$. Then $\pi_{n-1}(\partial D)$ is isotopic to a loop that is transverse to $\tau_{n-1}$ and covers at least one tao of $\tau_{n-1}$.
\end{lemma}

\begin{proof}
By construction, the link $L$ intersects the half-plane $\{(x,y,z)\ |\ z=0,\ y\geq n\}$ in $m$ arcs, each of which has its endpoints in the boundary of the half-plane and separates a disk from the half-plane. We will consider the case when $n$ is even and these disks are pairwise disjoint. For $n$ odd, one of the disks will contain the others and the argument is more complicated, but very similar. Let $D_1,\ldots, D_m$ be these disks and note that each $D_i$ intersects the plane $P_n$ in the arc $\alpha_i = \{(x,y,z)\ |\ z=0,\ y = n,\ 2i-1 \leq x \leq 2i\}$.

Let $D$ be a compressing disk above $P_n$ and assume $D$ is transverse to $D_1,\ldots, D_m$. Any loops of intersection between $D$ and each $D_i$ will determine $2$-spheres contained in the handlebody  $\mathcal{B}^+ \ssm K$  bounded by $P_n$. Since handlebodies are irreducible, such loops of intersection can be removed. Hence  $D \cap D_i$ is a collection of arcs for each $i$. Let $A \subset D$ be the intersection of $D$ with the entire collection of disks $\{\bigcup_i D_i\}$ and let $\beta \subset A$ be an outermost arc in $D$ bounding a disk $E \subset D$. Let $\gamma \subset \partial D \cap P_n$ be the arc of $\partial D$ that shares its endpoints with $\beta$ and let $\delta \subset \alpha_i$ be the arc in the appropriate $\partial D_i$ that shares its endpoints with $\beta$.

The arcs $\gamma$ and $\delta$ form a loop in $P_n$. If this loop bounds a disk with interior disjoint from the arcs $\alpha_i$ then this disk defines an isotopy that removes the arc $\beta$ from $D \cap \{\bigcup_i D_i\}$. Thus, if  this intersection has been minimized, the interior of the disk $E$ bounded by $\gamma \cup \delta$ must intersect one or more of the arcs $\alpha_i$. By construction, the arc $\gamma$ is disjoint from the arcs $\alpha_i$ away from its endpoints and the arc $\delta$ is contained in one of the arcs $\alpha_i$. So this implies that $E$ contains one or more of the arcs $\alpha_i$ in its interior.

Since each arc $\alpha_i$ is in the line $z = 0$ in $P_n$ and $\gamma \subset P_n$ is on the same side of this line near both its endpoints, this implies that the arc $\gamma$ must cross the line $z = 0$ at least twice. Moreover, if all the points of intersection are at values of $x$ greater than $2n$ or less than $1$ then $\gamma$ will be a trivial arc (possibly after passing it across the point at infinity, which corresponds to an isotopy in $\Sigma$.) This again contradicts the minimal intersection assumption. Therefore, a subarc of $\gamma$ must cut across one of the taos, as indicated in Figure~\ref{ttwaves}.

\begin{figure}[ht]
\includegraphics[width= 3.5in]{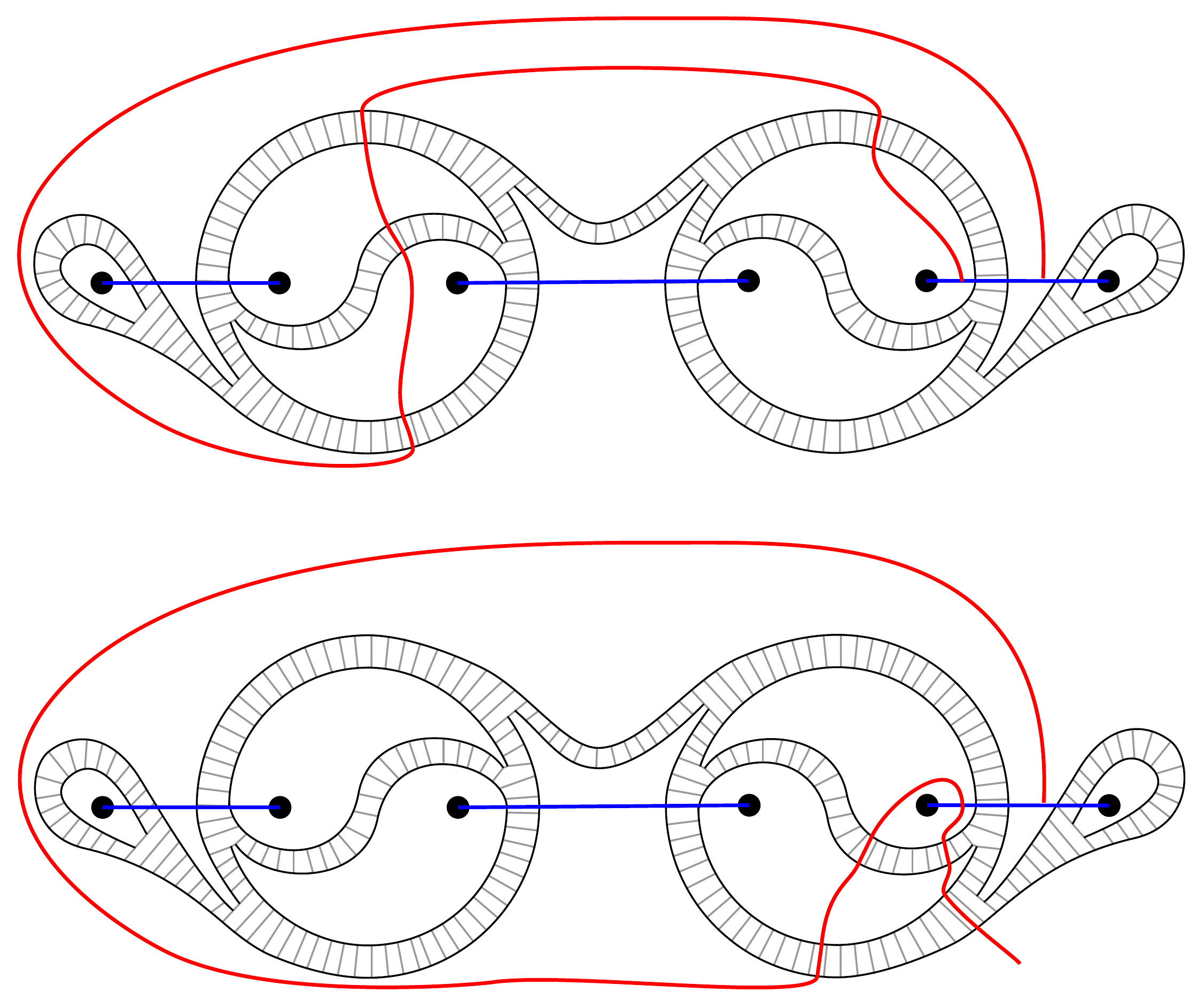}
\caption{Isotoping a transverse arc to a carried arc.}
\label{ttwaves}
\end{figure}

As suggested by  Figure ~\ref{ttwaves}, since $\gamma$ is disjoint from all $\alpha_i$'s, it can intersect the tao arc of this tao at most twice. Since $\partial D$ is a simple loop, we conclude that every arc of intersection between $\partial D$ and the tao disk $E_i$ intersects that tao arc at most twice. The reader can check that, as indicated in Figure~\ref{diskspin}, the image of any such arc under three or more half twists will be carried by this tao and cover it. Thus we can make the image $\pi_i(\partial D)$ transverse to $\tau_{n - 1}$ so that it covers at least one tao.
\end{proof}

\begin{figure}[ht]
\includegraphics[width= 3.5in]{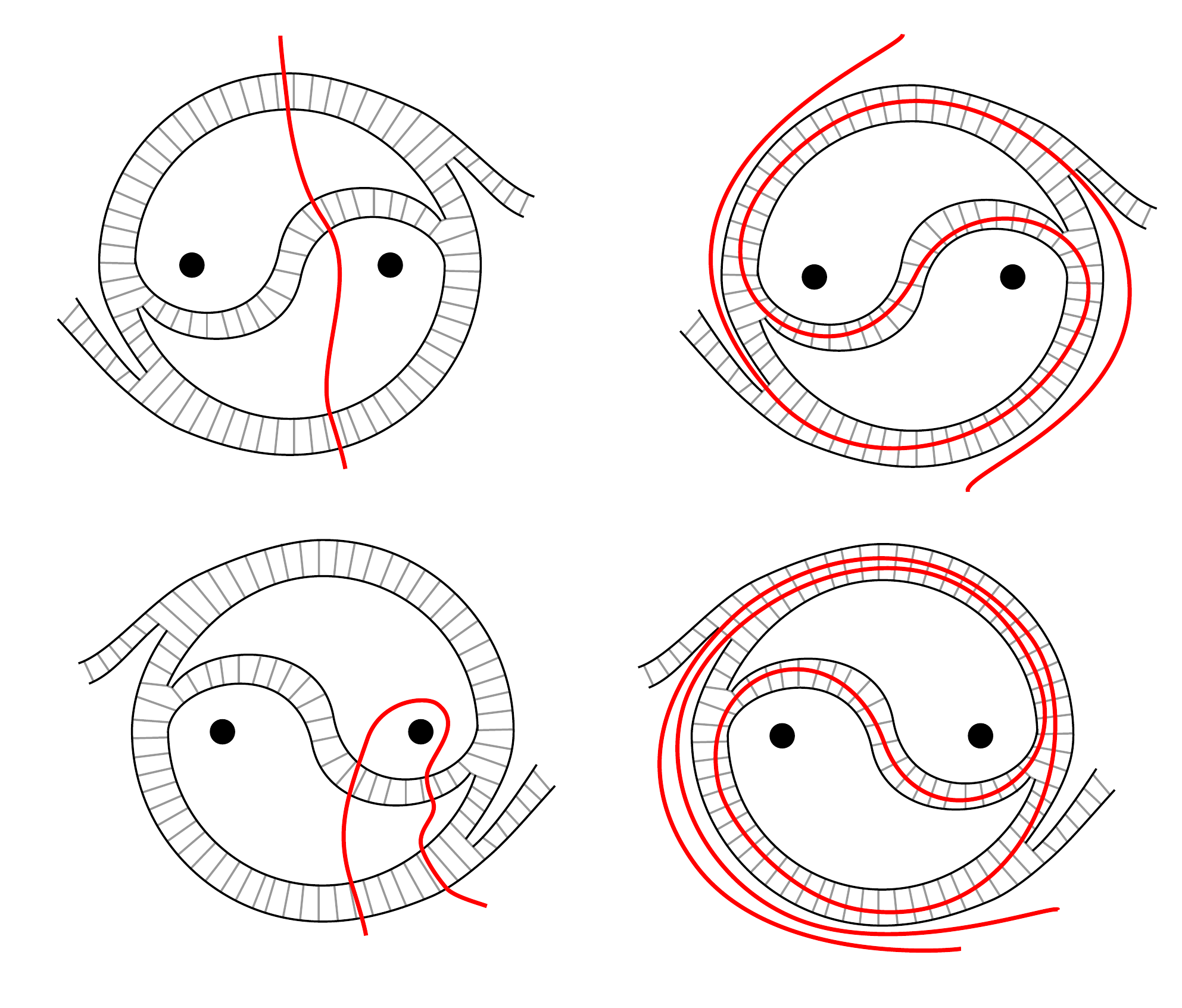}
\caption{Isotoping a transverse arc to a carried arc.}
\label{diskspin}
\end{figure}

\begin{lemma}
\label{diskcoversi}
Let $D \subset \mathcal{B}^+ \ssm K$ be a properly embedded, essential disk above the bridge surface $P_n$ and assume $n \geq 2(m - 2)$. Then for 
each  $k \leq 2(m - 2)$, the loop $\pi_{n-k}(\partial D)$ is isotopic to a loop that is transverse to  $\tau_{n - k}$  and covers 
$t$ taos of $\tau_{n - k}$ for some $t \geq ( k+1 )/2$. When $n - k$ is even and  $k > 1$,  $\ell$ covers $t$ taos and one eyelet.
\end{lemma}

\begin{proof}By the assumption above (as in Remark ~\ref{n- bounds}) $n \geq 2(m - 2)$ so the terms $\pi_{n-k}$ and $\tau_{n - k}$ are well defined. By Lemma~\ref{diskcoversone}, $\pi_{n-1}(\partial D)$ covers at least one tao $Y_{n-1,j}$ of $\tau_{n-1}$. By Lemma~\ref{LLCarried}, $\pi_{n - 2}(Y_{n - 1,j})$ covers either two taos of $\tau_{n - 2}$ or one tao and one eyelet. (The latter occurs only when $n - 2$ is even.) Therefore, for each arc $\gamma$ of $\partial D$ such that $\pi_{n-1}(\gamma)$ covers $Y_{n-1,j}$, we conclude that $\pi_{n - 2}(\gamma)$ covers either two taos of $\tau_{n-2}$ or one tao and one eyelet. If we repeat this argument, noting that the latter case can happen at most half of the time, we conclude that $\pi_{n-k}(\partial D)$ covers at least $k/2$ taos of $\tau_{n-k}$ or $k/2-1$ taos and one eyelet.
\end{proof}

\begin{corollary}
\label{< r almost carried}
Let $D$ be a properly embedded, essential disk above the bridge surface $P_n$ and $r$ a positive integer so that  $t = n - 2r(m - 2) \geq 1 $ for $n$ odd and $t = n - 2r(m - 2)+1 \geq 1$ for $n$ even. Then for every loop $\ell \subset \Sigma$ such that $d(\partial D, \ell) < r$, $\pi_t(\ell)$ is almost carried by $\tau_t$.
\end{corollary}

\begin{proof}
Since $t \geq 1$ the restrictions on $r$ are necessary. We will prove this by induction on $r$. Consider the base case $r = 1$:  Since $d(\partial D, \ell) = 0$ the loop $\ell $ is isotopic to $\partial D$.  In this case  $t = n - 2(m - 2)$ when $n$ is  odd and $t = n - 2(m - 2) + 1$ when $n$ is even. In either case, $t$ is be odd, so the train track  $\tau_{n - 2(m - 2)}$ has exactly $m - 1$ taos. By Lemma~\ref{diskcoversi} applied to the case $k = 2(m - 2)$, we have that $t \geq m - 1$ so  $\pi_t(\ell)$ covers all the $m - 1$ taos and one eyelet of $\tau_t$. Since $\pi_t(\partial D)$ is transverse to $\tau_t$, by Lemma~\ref{diskcoversi} all the non-carried arcs of $\pi_t(\partial D) \cap \tau_t$ must be along interval fibers. But since $\pi_t(\partial D)$ is simple, this implies that $\pi_t(\partial D)$ is almost carried by $\tau_t$, completing the proof of the base case.

For the inductive step, assume $d(\partial D, \ell) < r$, i.e. that there is a path in $\Sigma$ from $\partial D$ to $\ell$ of length strictly less than $r$. Then the vertex right before $\ell$ in this path will represent a loop $\ell'$ such that $d(\partial D, \ell') < r - 1$. By the inductive hypothesis, this means that for $t' = n - 2(r - 1)(m - 2)$ for $n$ odd and $t' = n - 2(r - 1)(m-2)+1$ for $n$ even, the loop $\pi_{t'}(\ell')$ is almost carried by $\tau_{t'}$. Note that regardless of the parity of $n$, both $t$ and $t'$ are odd and $t' - t = 2(m-2)$.

Because $\pi_{t'}(\ell')$ is almost carried by $\tau_{t'}$, it must bisect at least one tao arc by Lemma~\ref{hitsonetao}. By repeatedly applying Lemma~\ref{up the ladder}, we find that $\pi_t(\ell')$ covers $1 + \frac{1}{2}(t' - t) = 1 + m-2 = m-1$ taos and one eyelet of $\tau_t$. Since $\pi_t(\ell)$ is disjoint from $\pi_t(\ell')$, Lemma~\ref{coversdisjointcarried} implies that $\pi_t(\ell)$ is also almost carried by $\tau_t$. This completes the induction step and thus the proof.
\end{proof}

The final step in the proof that the constant given by Theorem~\ref{main theorem} is in fact a lower bound is the following Lemma:

\begin{lemma}
\label{lower disk}
If $D \subset \mathcal{B}^- $ is a properly embedded, essential disk below $P_1$ then $\partial D$ is not almost carried by $\tau_1$.
\end{lemma}

\begin{proof} 
The proof is similar to the first part  of the proof of Lemma~\ref{diskcoversone}. We bring  it here for the sake of completeness. By construction, the link $K$ intersects the half-plane $\{(x,y,z)\ |\ z=0,\ y\leq 1\}$ in $m$ arcs, each of which has its endpoints in the boundary of the half-plane and separates a disk from the half-plane. Denote these disks by  $D_1,\ldots, D_m$ and note that $D_i \cap P_n$ is an arc $\alpha_i = \{(x,y,z)\ |\ z=0,\ y = 1,\ 2i-1 \leq x \leq 2i\}$,  for each $i$.

Let $D$ be a compressing disk below $P_1$ and assume $D$ is transverse to $D_1,\ldots, D_m$ and minimizes the intersections with them. Any loops of intersection between $D$ and each $D_i$ will determine $2$-spheres contained in the handlebody  $\mathcal{B}^- \ssm K$  bounded by $P_1$. Since handlebodies are irreducible, such loops of intersection can be removed. Hence  $D \cap D_i$ is a collection of arcs for each $i$. 

Set $A = D \cap \{\cup D_i\}$ and let $\beta \subset A$ be an outermost arc in $D$ bounding a disk $E \subset D$ so that $\partial E = \beta \cup \gamma$ where $\gamma \subset P_1$. Let $\delta \subset \alpha_i$ be the arc in the appropriate $\partial D_i$ that shares its endpoints with $\beta$.

The arcs $\gamma$ and $\delta$ form a loop in $P_1$. This loop cannot bound a disk with interior disjoint from the arcs $\alpha_i$ as this would define an isotopy that removes the arc $\beta$ from $D \cap \{\cup D_i\}$ contradicting the choice of $D$. Thus, the interior of the disk $E$ bounded by $\gamma \cup \delta$ must intersect one or more of the arcs $\alpha_i$. By construction, the arc $\gamma$ is disjoint from the arcs $\alpha_i$ away from its endpoints and the arc $\delta$ is contained 
in one of the arcs $\alpha_i$. So this implies that $E$ contains one or more of the arcs $\alpha_i$ in its interior.

Since each arc $\alpha_i$ is in the line $z = 0$ in $P_1$ and $\gamma \subset P_1$ is on the same side of this line near both its endpoints, this implies that the arc $\gamma$ must cross the line $z = 0$ at least twice. Moreover, if all the points of intersection are at values of $x$ greater than $2n$ or less than $1$ then $\gamma$ will be a trivial arc (possibly after passing it across the point at infinity, which corresponds to an isotopy in $\Sigma$.) This again contradicts the minimal intersection assumption. Therefore, a subarc of $\gamma$ must cut across one of the taos transversely, as indicated in Figure~\ref{ttwaves}.  However note that the $\alpha_i$'s ``block" the access to each switch. More precisely they contain subarcs which together with arcs emanating from the switch bound either a triangle or a quadrilateral with the switch pointing into its interior. This implies that 
$\gamma$ cannot  enter or leave $\tau_1$ at a switch and therefor cannot be almost carried by it.
\end{proof}

\begin{corollary}
\label{lower bound} 
Given a link $K \subset S^3$ defined by a highly twisted $n$-row, $2m$-plat presentation and the induced bridge surface $\Sigma$, $d(\Sigma) \geq  \lceil n /(2(m - 2)) \rceil$.
\end{corollary}

\begin{proof} 
Assume for contradiction $d(\Sigma) < \lceil n /(2(m - 2)) \rceil$. Then there is a loop $\ell \subset P_n$ bounding a disk above $P_n$ and a loop $\ell' \subset P_1$ bounding a disk below $P_1$ such that $d(\pi_1(\ell), \ell') < \lceil n /(2(m - 2)) \rceil$. By Lemma~\ref{< r almost carried}, the distance bound implies that $\ell'$ is almost carried by $\tau_1$. However, Lemma~\ref{lower disk} implies that $\ell'$ cannot be carried by $\tau_1$. This contradiction implies that $d(\Sigma) \geq \lceil n /(2(m - 2))  \rceil$. 
\end{proof}

\begin{proof}[Proof of Theorem \ref{main theorem}] 
Combining the upper bound in Lemma~\ref{upper bounds on d} with the lower bound in Corollary ~\ref{lower bound} gives us the equality $d(\Sigma) = \lceil n /(2(m - 2))  \rceil$.
\end{proof}

\bibliographystyle{amsplain}
\bibliography{plats}

\end{document}